\documentclass{amsart}

\usepackage{marktext}
\usepackage{gimac}

\DeclareMathOperator{\Leb}{Leb}

\DeclarePairedDelimiter{\ip}{\langle}{\rangle}

\begin{document}
% Title etc.

\title[Quantitative Mixing Rate of the Pierrehumbert Flow]
{Quantitative   Dependence of the Pierrehumbert Flow's Mixing Rate on the Amplitude}

\author[Son]{Seungjae Son}
\address{%
  Department of Mathematical Sciences, Carnegie Mellon University, Pittsburgh, PA 15213.
}
\email{seungjas@andrew.cmu.edu}

\begin{abstract}
  We quantitatively study the mixing rate of randomly shifted alternating shears on the torus.
  This flow was introduced by Pierrehumbert '94, and was recently shown to be exponentially mixing.
  In this work, we quantify the dependence of the exponential mixing rate on the flow amplitude.
  Our approach is based on constructing an explicit Lyapunov function and a coupling trajectory for the associated two-point Markov chain, together with an application of the quantitative Harris theorem.
%The Pierrehumbert flow~\cite{Pierrehumbert94} is an alternating sine shear flow on the torus, and it was shown in~\cite{BlumenthalCotiZelatiEA23} to exhibit exponential mixing.
\end{abstract}

\maketitle

\section{Introduction}

\subsection{Main Results}
We begin by stating the main result of this paper. Let $\T \cong \R/2\pi\Z$ and $(\zeta_n)_{n\geq 1}$ be a sequence of i.i.d random variables that are uniformly distributed on $\T$. Define a random velocity field $u:[0,\infty) \times \T^2 \to \R^2$ such that for each $n\in \N_0$,
\begin{equation}\label{e:udef}
    u(t,x) = 
\begin{cases}
A\sin(x_2-\zeta_{2n+1})e_1\,, & t\in [2n, 2n+1)\\
A\sin(x_1-\zeta_{2n+2})e_2\,, & t\in [2n+1, 2n+2)
\end{cases}\,,
\end{equation}
where $x=(x_1, x_2)\in \T^2$ and $e_i$ is the $i$-th standard basis vector. 

\begin{theorem}\label{t:main-theorem}
    Fix $q<\infty$. For all sufficiently large $A$, there exists a random $D_A>0$ such that every solution to the transport equation
    \begin{equation}\label{e:transport}
        \partial_t \phi + u\cdot \nabla \phi = 0\,,
    \end{equation}
    with mean-zero initial data $\phi_0 \in H^1(\T^2)$ satisfies
    \begin{equation}\label{e:mixing}
        \norm{\phi_n}_{H^{-1}} \leq D_A \exp(-\exp(-A^{96})n)\norm{\phi_0}_{H^1}\,, \forall n\in 2\N\,, \quad\text{a.s}\,.
    \end{equation}
    Moreover, there exists a finite deterministic constant $\bar D_q$, that is independent of $A$, such that 
    \begin{equation}\label{e:Dmoment-bound}
        \E[D_A^q] \leq \bar D_q\,.
    \end{equation}
    Here, $u$ is defined as in~\eqref{e:udef} and $\E$ is the expectation with respect to the probability measure where $(\zeta_n)_{n\geq 1}$ is defined.
\end{theorem}

\begin{remark}
   Combining the argument in~\cite[Theorem~1.1, Step~2]{CoopermanIyerEA25} with the $C^1$-regularity of the flow map, one can extend Theorem~\ref{t:main-theorem} to show that the mixing estimate~\eqref{e:mixing} holds starting from any time $m \in 2\mathbb{N}$ and for all flow times $t \ge 0$, at the expense of a time-dependent prefactor. More precisely, for all sufficiently large $A$, there exists a finite random variable $D_A' > 0$ such that, almost surely, for all $m \in 2\mathbb{N}$ and $t \ge 0$,
\begin{equation}
    \|\phi_{m+t}\|_{H^{-1}} 
    \le A^2 D_A' \,(m^2 + 1)\,
    \exp\paren*{- \exp(-A^{96}) \,\floor*{t/2}}\,
    \|\phi_m\|_{H^1}.
\end{equation}
In addition, $D_A'$ has $A$-independent finite moments; that is,
\begin{equation}
    \mathbb{E}[D_A'] \le \bar D < \infty\,,
\end{equation}
for some constant $\bar D$ independent of $A$.

We do not include this extension in the statement or proof of the theorem, as our main focus is on the mixing \emph{rate} rather than on optimizing the prefactor in front of the decay estimate.
\end{remark}

\subsection{Motivation and Background}
We now motivate the problem and introduce the relevant background to place our result in the context of the existing literature.

The alternating sine–shear flow introduced by Pierrehumbert in~\cite{Pierrehumbert94} has become a canonical model in the study of chaotic advection and tracer microstructure formation. Its appeal lies in its simplicity and its ability to generate stretching, folding, and small scales, making it an ideal setting to study mixing phenomena in geophysical and engineering contexts (see, e.g.,\cite{Beron2009, ArefEA17}).

The chaotic dynamics of such flows can be rigorously analyzed through \emph{mixing}. Broadly speaking, mixing describes the creation of finer scales and larger gradients in transported scalar fields, and its mathematical analysis has been the focus of extensive research~\cite{ConstantinKiselevEA08, LinThiffeaultEA11, Seis13, MilesDoering18}. Explicit constructions of rapidly mixing velocity fields are available in works such as~\cite{YaoZlatos17, AlbertiCrippaEA19, ElgindiZlatos19, BedrossianBlumenthalEA22, ElgindiLissEA25}. Moreover, mixing plays a central role in a variety of applications: it underlies enhanced dissipation~\cite{BedrossianBlumenthalEA21, CotiZelatiDelgadinoEA20, CoopermanIyerSon25, TaoZworski25, FengIyer19}, the suppression of singularity formation in nonlinear PDEs~\cite{KiselevXu16, BedrossianHe17, IyerXuEA21, He23}, and rapid convergence of stochastic diffusions~\cite{ChristieFengEA23}.

In connection with the Pierrehumbert flow, numerical evidence in~\cite{ChengRajasekaranEA23} suggested that a closely related sawtooth shear exhibits fast (geometric) mixing. This was rigorously confirmed in~\cite{BlumenthalCotiZelatiEA23}, which proved exponential mixing for the Pierrehumbert flow. Specifically, \cite{BlumenthalCotiZelatiEA23} shows that there exists $\gamma>0$ and a finite random variable $D>0$ such that, almost surely,
\begin{equation}\label{e:exp-mix-pierrehumbert}
    \|\phi_t\|_{H^{-1}}
    \le D\, e^{-\gamma t}\, \|\phi_0\|_{H^1},
    \qquad \forall t \ge 0,
\end{equation}
for every $\phi$ that solves~\eqref{e:transport} with velocity field $u$ defined in~\eqref{e:udef}.
The mixing rate~$\gamma$, however, is not explicit as the proof in~\cite{BlumenthalCotiZelatiEA23} produces it through arguments of Furstenberg.

In this paper we are interested in quantifying the dependence of~$\gamma$ on the flow amplitude~$A$.
%A natural question that follows is how the exponential mixing rate~$\gamma$ depends on parameters of the flow, such as the shear amplitude~$A$.
Numerical simulations in~\cite{ChengRajasekaranEA23} indicate that the effective mixing rate increases with~$A$, i.e.\ one expects $\gamma(A)\to\infty$ as $A\to\infty$. Related results in~\cite[Theorem~2]{Cooperman23} show that for deterministic phase shifts and random switching times, the logarithmic mixing scale increases at least linearly in~$A$. However, rigorously capturing this dependence in the Pierrehumbert setting appears to be challenging.

The main contribution of this paper is to make one step in this direction.
Indeed, the quantitative estimate~\eqref{e:mixing} we prove yields $\gamma(A)=\exp(-A^{96}) \to 0$ as $A\to\infty$.
This is weaker than the expected growth, and the reason for this is mainly
%This discrepancy arises
due to the difficulty in tracking how successive orthogonal shears interact at high amplitudes. For example, a thin vertical strip is stretched to scales of order $A^{-1}$ after one horizontal shear, but following the vertical shear, the resulting structure becomes highly irregular and spatially inhomogeneous. It is unclear whether scales of order $A^{-2}$ are produced uniformly, and this loss of coherence between shearing steps prevents a straightforward iteration of $A$-dependent mixing gains.

To overcome this challenge, we develop a new approach that makes the dependence on~$A$ explicit. Rather than tracing constants implicitly through the proof of~\cite{BlumenthalCotiZelatiEA23}, which relies on qualitative tools such as equidistribution of irrational rotations, we construct explicit Lyapunov functions and coupling trajectories for the associated two-point Markov process. A key step is the identification of a family of ``small sets'' in phase space whose properties remain uniformly controlled in~$A$, allowing for the application of a quantitative Harris theorem. To the best of our knowledge, this yields the first explicit quantitative upper bound on the mixing rate for the Pierrehumbert flow.

Finally, we note that the mixing rate $\gamma(A)$ cannot grow faster than linearly in~$A$ (see, for example,~\cite{IyerKiselevEA14}). Taking this upper bound, the numerical evidence in~\cite{ChengRajasekaranEA23}, and heuristic insights from~\cite{Cooperman23} into account, it is natural to ask whether one can improve upon~\eqref{e:mixing} to obtain sharper estimates or even saturate the optimal scaling $\gamma(A)\sim A$.

\subsection*{Plan of the Paper}

In Section~\ref{s:notations}, we introduce the notation and preliminaries. Section~\ref{s:statement-of-harris} states the Harris conditions in Lemmas~\ref{l:Lyapunov-drift} and~\ref{l:Harris-conditions}.

In Section~\ref{s:proof-of-r1-uniform-small}, we prove the uniform small set property of the set $R$ (Lemma~\ref{l:R1-uniform-small}), relying on Lemma~\ref{l:z-smallset}. The proof of Lemma~\ref{l:z-smallset} is given in Section~\ref{s:quantitative-ift-ift}, using quantitative versions of the inverse and implicit function theorems.

Section~\ref{s:verify-Harris} is devoted to proving the Harris conditions. In Section~\ref{s:off-diagonal-small}, we state and prove Lemmas~\ref{l:smallset-to-R1} and~\ref{l:R_1-small} to establish the small set property of the off-diagonal set (Lemma~\ref{l:final-small-set}). Using this, we then prove Lemmas~\ref{l:Lyapunov-drift} and~\ref{l:Harris-conditions} in Section~\ref{s:proof-of-harris-lemma}.

Finally, Section~\ref{s:proof-of-main-theorem} proves the main result, Theorem~\ref{t:main-theorem}, by applying Lemma~\ref{l:Harris-conditions} and the quantitative Harris theorem.

\section*{Acknowledgement}
The author thanks Gautam Iyer for his helpful suggestions and advice.

\section{Statement of Harris Conditions}

In this section, we introduce the notation and necessary preliminaries, and then state Lemma~\ref{l:Lyapunov-drift} and Lemma~\ref{l:Harris-conditions}, which provide the Harris conditions required to prove Theorem~\ref{t:main-theorem} in Section~\ref{s:proof-of-main-theorem}.

\subsection{Notation and Preliminaries}\label{s:notations}
\subsection*{Random Dynamical System}

Let $\xi^m \defeq (\xi_1, \ldots, \xi_{2m}) \in \T^{2m}$, and define the velocity field 
$u_m: [0, 2m) \times \T^2 \to \R^2$ by
\begin{equation}\label{e:umdef}
    u_m(t, x; \xi^m) =
    \begin{cases}
        A \sin(x_2 - \xi_{2n+1}) e_1, & t \in [2n, 2n+1), \\[2mm]
        A \sin(x_1 - \xi_{2n+2}) e_2, & t \in [2n+1, 2n+2),
    \end{cases}
\end{equation}
and consider the corresponding flow $\Phi: [0, 2m) \times \T^2 \to \T^2$ defined by
\begin{equation}
    \partial_t \Phi(t, x; \xi^m)
    = u_m\big(t, \Phi(t, x; \xi^m); \xi^m\big), 
    \qquad \Phi(0, x; \xi^m) = x.
\end{equation}

Given a random sequence $(\zeta_n)_{n \geq 1}$ taking values in $\T$, define 
$\underline{\zeta}^m \defeq (\zeta_1, \ldots, \zeta_{2m})$. 
Then the sequence $(\Phi(2m, x; \underline{\zeta}^m))_{m \ge 1}$ forms a random dynamical system satisfying
\begin{equation}
    \Phi(2m+2, x; \underline{\zeta}^{m+1}) 
    = \Phi\big(2,\, \Phi(2m, x; \underline{\zeta}^m);\, (\zeta_{2m+1}, \zeta_{2m+2})\big)\,.
\end{equation}

Moreover, the corresponding two-point process
\begin{equation}
    \Phi^{(2)}(2m, (x,y); \underline{\zeta}^m) 
    = \big(\Phi(2m, x; \underline{\zeta}^m),\, \Phi(2m, y; \underline{\zeta}^m)\big),
    \qquad (x,y)\in \T^{2,(2)},
\end{equation}
defines a random dynamical system on the product space
\begin{equation}
    \T^{2,(2)} \defeq \T^2 \times \T^2 \setminus \Delta, 
    \qquad
    \Delta \defeq \{(x,y) \in \T^2 \times \T^2 \mid x = y\}.
\end{equation}
Here, the diagonal $\Delta$ is removed to exclude the trivial stationary measure of the two-point chain.

Throughout the paper, we write $z = (x,y)$ for points in $\T^{2,(2)}$ with $x, y \in \T^2$, and we often suppress the explicit dependence of $\Phi$ on deterministic $\xi^m$ or random $\underline{\zeta}^m$ when it is clear from context. For brevity, we also write
\[
\Phi_n(x), \qquad \Phi_{2n}^{(2)}(z, \underline{\zeta}^n), \qquad \text{or} \qquad \Phi_{2n}^{(2)}(z, \xi^n).
\]

\subsection*{Two-point Markov Process}

From the independence of $(\zeta_n)_{n\geq1}$, we see that the process 
$(\Phi_{2m}^{(2)}(z))_{m\geq 1}$, defined for any $z\in \T^{2,(2)}$, 
is a Markov chain taking values in $\T^{2,(2)}$. 
In particular, we define the notation
\begin{equation}
    P_A^{(2), n}(z, E) \defeq \P\brak{\Phi_{2n}^{(2)}(z) \in E},
    \qquad \forall n \in \N,
\end{equation}
to denote the $n$-step transition kernel of the Markov process $\Phi^{(2)}$.

We also recall the notion of a \emph{small set}, which will be used throughout the paper 
in proving the ergodicity of the chain $P_A^{(2)}$. 
\begin{definition}
Let $(E, \mathcal E)$ be a measurable space, and let 
$P: E \times \mathcal E \to [0,1]$ be a Markov transition kernel. 
A set $S \in \mathcal E$ is called \emph{small} if there exist $\alpha \in (0,1]$ 
and a probability measure $\nu$ on $(E, \mathcal E)$ such that
\begin{equation}
    \inf_{z \in S} P(z, \cdot) \geq \alpha\, \nu(\cdot).
\end{equation}
Here, $\nu$ is referred to as a \emph{minorizing measure} for $P$.
\end{definition}

\subsection{Harris Conditions}\label{s:statement-of-harris}

In this subsection, we state the Harris conditions used to prove Theorem~\ref{t:main-theorem}. 
These consist of a Lyapunov drift condition and a small set condition for the two-point Markov chain~$P_A^{(2)}$, formulated respectively in Lemma~\ref{l:Lyapunov-drift} and Lemma~\ref{l:Harris-conditions}. 
Once these are established, the quantitative Harris theorem~\cite{HairerMattingly11} implies the geometric ergodicity of~$P_A^{(2)}$, and hence the almost sure mixing of the Pierrehumbert flow, with explicit dependence on the amplitude~$A$.

The main contribution of this paper is the quantitative verification of the small set property in Lemma~\ref{l:Harris-conditions}, where we track the dependence of all constants on~$A$. 
Together with the Lyapunov drift condition, this provides a complete and quantitative proof of the Harris conditions required for Theorem~\ref{t:main-theorem}.

\begin{lemma}\label{l:Lyapunov-drift}
    There exist $A$-independent constants $p>0, \gamma\in (0,1)$, $K_1>0$, and a function $V\colon \T^{2, (2)}\to [1,\infty)$ such that for all sufficiently large $A$, 
    \begin{equation}\label{e:lyapunov-one-step}
        P_A^{(2)} V \leq \gamma V + K_1A^{2p}\,.
    \end{equation}
\end{lemma}

\begin{lemma} \label{l:Harris-conditions}
    Let $p, \gamma, K_1, V$ be as in Lemma~\ref{l:Lyapunov-drift}. Then, for all sufficiently large $A$, there exist $A$-dependent constants $R_2 = 2CK_1A^{2p} / (1-\gamma)$, $M \leq CA^{95}$, and a probability measure $\nu_A$ on $\T^{2, (2)}$ such that we have
    \begin{gather}
        \label{e:lyapunov-multi-step}
        P_A^{(2), M}V \leq \gamma^M V + K_1 A^{2p} \frac{1-\gamma^M}{1-\gamma}\,,\\
        \label{e:sublevel-V-small}
        \inf_{z\in \set{V \leq R_2}} P_A^{(2), M}(z, \cdot) \geq A^{-CA^{95}} \nu_A(\cdot)\,.
    \end{gather}
    Here, $C>1$ is an $A$-independent constant.
\end{lemma}

For continuity, we defer the proofs of Lemma~\ref{l:Lyapunov-drift} and Lemma~\ref{l:Harris-conditions} to Section~\ref{s:verify-Harris}. In particular, to prove the small set condition~\eqref{e:sublevel-V-small} 
in Lemma~\ref{l:Harris-conditions}, we require a key intermediate result—%
Lemma~\ref{l:R1-uniform-small}—which establishes the uniform small set 
property of a specific subset \(R \subset \T^{2,(2)}\). 
Since this lemma forms one of the central technical contributions of the paper, 
we first state and prove it separately in Section~\ref{s:proof-of-uniform-small-set}. 
Afterward, we proceed to prove the auxiliary lemmas needed for 
Lemmas~\ref{l:Lyapunov-drift} and~\ref{l:Harris-conditions}, 
and finally prove both in Section~\ref{s:verify-Harris}.

\section{Uniform Small Set Property of the Set \(R\)}\label{s:proof-of-uniform-small-set}

In this section, we prove Lemma~\ref{l:R1-uniform-small}, 
which establishes the uniform small set property of the set
\begin{equation}\label{e:Rdef}
    R \defeq \{(x, y) \in \T^{2,(2)} : y - x = (\pi, \pi)\}.
\end{equation}
This lemma plays a key role in the proof of 
Lemma~\ref{l:Harris-conditions} in Section~\ref{s:verify-Harris} 
and provides estimates that remain quantitatively tractable 
in terms of the amplitude~\(A\).

Before proceeding with the proof, we introduce notation that will be used throughout the remainder of the paper. 
We write \(z = (x, y)\) for points in \(\T^{2,(2)}\) with \(x, y \in \T^2\), and denote the coordinates of \(x\) by \(x_1, x_2 \in \T\). 
The symbols \(\xi\) and \(\zeta\) represent deterministic and random points in~\(\T\), respectively, and the superscript~\(n\) indicates that these are \(2n\)-dimensional vectors:
\begin{equation}
    \xi^n,\, \underline{\zeta}^n \in \T^{2n}\,.
\end{equation}
We also denote by \(B_\infty(z, r)\) the \(\ell^\infty\)-ball of radius~\(r\) centered at~\(z\).

We now state the main result of this section.

\begin{lemma}\label{l:R1-uniform-small}
    There exists $A$-independent constant $C_1, C_2, C_3, C_4>0$ such that for any $z\in R$ and $n\geq 2$,
    \begin{equation}\label{e:zR-uniform-small}
        \inf_{z'\in B_\infty(z, r_1(A))} P_A^{(2), n}(z', \cdot) \geq c_1(A)^{\floor{\frac{n}{2}}-1} c_2(A) \Leb|_{B_\infty(z, r_2(A))}\,.
    \end{equation}
    Here, $r_1, r_2, c_1$, and $c_2$ are $A$-dependent but $z, n$-independent constants defined by
    \begin{align}
        r_1(A) = C_1A^{-95}\,,\quad r_2(A) = C_2 A^{-37}\,,\quad c_1(A)=C_3 A^{-378}\,,\quad c_2(A)=C_4 A^{-196}\,.
    \end{align}
\end{lemma}

\subsection{Proof of Lemma~\ref{l:R1-uniform-small}}\label{s:proof-of-r1-uniform-small}

The following lemma serves as the key ingredient in the proof of Lemma~\ref{l:R1-uniform-small}. For clarity, we defer its proof to Section~\ref{s:quantitative-ift-ift} and begin by proving Lemma~\ref{l:R1-uniform-small} below.
\iffalse
We note that, by~\eqref{e:z_*-small}, the support of the minorizing measure is a ball concentric with the small set itself. 
This allows us to iterate the estimate and, using the Markov property, deduce that the small set property holds for any $P_A^{(2), n}$, with $n \geq 2$.
\fi

\begin{lemma}\label{l:z-smallset}
Suppose that $w \in \T^{(2), 2}, \xi \in \T^4$, and $s>0$ satisfy
\begin{enumerate}
    \item \label{i:wxi-fp} $\Phi_4^{(2)}(w, \xi) = w$.
    \item \label{i:wxi-lb} $\abs{\det D_\zeta \Phi_4^{(2)}(w, \xi)} \geq s$.
\end{enumerate}
Then, there exists a constant $C$, independent of $A, w, \xi$, and $s$, such that
    \begin{equation}\label{e:z_*-small}
        \inf_{z\in B(w, CA^{-105}s^2)}P_A^{(2), 2}(z,\cdot) \geq Cs^{-1}\Leb|_{B(w, CA^{-49}s^2)}\,.
    \end{equation}
\end{lemma}

\begin{proof}[Proof of Lemma~\ref{l:R1-uniform-small}]
By using the definition of the Pierrehumbert flow,
we observe that for any $z \in R$, the choice of $\xi^1(z)=\xi^1(x,y)=(x_2,x_1)$ yields
\begin{equation}\label{e:phi2-fp}
    \Phi_2^{(2)}(z, \xi^1(z)) = z\,.
\end{equation}
This implies
\begin{equation}\label{e:zxi-fp}
    \Phi_4^{(2)}(z, \xi^2(z)) = \Phi_2^{(2)}(\Phi_2^{(2)}(z, \xi^1(z)), \xi^1(z)) = z\,,
\end{equation}
where $\xi^2$ is a function defined as $\xi^2(z)=\xi^2(x,y)=(x_2,x_1,x_2,x_1)$.

Moreover, using the definition of the Pierrehumbert flow again, we notice that for any $z_1, z_2 \in R$ and $\xi \in \T^4$, 
\begin{equation}
    \Phi_4^{(2)}(z_1, \xi+\xi^2(z_1)) - z_1 = \Phi_4^{(2)}(z_2, \xi+\xi^2(z_2)) - z_2\,,
\end{equation}

Thus, differentiating both sides of the identity with respect to $\xi$ and setting $\xi = 0, z_1=z_*=((0,0),(\pi, \pi))$ yields that for any $z\in R$,
\begin{equation}\label{e:dzeta-same}
    D_{\xi}\Phi_4^{(2)}(z, \xi^2 (z)) = D_{\xi}\Phi_4^{(2)}(z_*, \xi^2(z_*))\,.
\end{equation}

In particular, for any $z\in R$, \eqref{e:zxi-fp}, \eqref{e:dzeta-same}, and~\eqref{e:det-dzeta-z*} satisfy the conditions (\ref{i:wxi-fp}) and (\ref{i:wxi-lb}) of Lemma~\ref{l:z-smallset} with $(w, \xi, s)=(z, \xi^2(z), 4A^6)$ and hence, \eqref{e:z_*-small} implies that for some constant $C$, independent of $z$ and $A$,
\begin{equation}
    \inf_{z'\in B(z, CA^{-93})}P_A^{(2), 2}(z',\cdot) \geq CA^{-6}\Leb|_{B(z, CA^{-37})}\,.
\end{equation}

Fix $z\in R$. Hereafter, we denote $C$ as a generic constant whose value may vary from line to line but remains independent of $z$ and $A$.
For the notational brevity, let $\mu$ be a measure on $\mathcal B(\T^{(2), 2})$ such that 
\begin{equation}\label{e:mudef}
\mu = CA^{-6}\Leb|_{B(z, CA^{-37})}\,.    
\end{equation}
Then we have $\mu(B(z,CA^{-93}))=C A^{-93*4-6}= C A^{-378} \defeq c_1(A)$. 

By using Markov property and a simple induction, we see that for any $n\geq 1$, 
\begin{equation}\label{e:even-small}
    \inf_{z'\in B(z, CA^{-93})} P_A^{(2), 2n}(z',\cdot) \geq c_1(A)^{n-1}\mu(\cdot)\,.
\end{equation}

Moreover, for any $z'\in B(z, CA^{-95})$ and $\underline{\zeta}^1 \in B(\xi^1(z), CA^{-95})$, using~\eqref{e:dzeta}, \eqref{e:dx}, and~\eqref{e:phi2-fp} yields
\begin{equation}
    \abs{\Phi_2^{(2)}(z', \underline{\zeta}^1) - z} = \abs{\Phi_2^{(2)}(z', \underline{\zeta}^1) - \Phi_2^{(2)}(z, \xi^1(z))} \leq CA^{-93}\,,
\end{equation}
and hence
\begin{equation}
    P_A^{(2), 2}(z', B(z, CA^{-93})) \geq \P[\underline{\zeta}^1 \in B(\xi^1(z), CA^{-95})] = CA^{-190}\,.
\end{equation}
Combining this with~\eqref{e:even-small} and using the Markov property, we get that for any $n \geq 2$,
\begin{equation}
    \inf_{z'\in B(z, CA^{-95})} P_A^{(2), n}(z', \cdot) \geq c_1(A)^{\floor{\frac{n}{2}}-1} CA^{-190} \mu(\cdot)\,.
\end{equation}
Finally, by substituting the definition of $\mu$ from~\eqref{e:mudef} and noting that the balls $B$ defined using the Euclidean norm are equivalent to the $B_\infty$ balls (up to a constant change in radius), we conclude the proof for the uniform small set property~\eqref{e:zR-uniform-small}.
\end{proof}

\subsection{Proof of Lemma~\ref{l:z-smallset}}\label{s:quantitative-ift-ift}

Lemma~\ref{l:z-smallset} is a direct application of the quantitative inverse and implicit function theorems, stated below. In particular, if a family of functions $f$ shares the same $C^2$-norm and Jacobian determinant at a point, these theorems provide estimates that are uniform across the family.

We first prove Lemma~\ref{l:z-smallset} using Lemma~\ref{l:quantitative-ift} and Lemma~\ref{l:quantitative-ift2}, and then reproduce the proofs of these two lemmas, carefully tracing how the constants depend on the $C^2$-norm and the Jacobian determinant of each function.

\begin{lemma}[Quantitative Inverse Function Theorem~\cite{Christ85}]\label{l:quantitative-ift}
Let $F\colon \R^d \to \R^d$ be a $C^2$ function. If we let $D, r$ be two positive constants such that
\begin{equation}\label{e:Drassumption}
    D \geq \norm{F}_{C^2} \quad\text{and}\quad r\leq \abs{\det DF(0)}\,,
\end{equation}
then there exists constants 
\begin{equation}\label{e:Cdefs}
    C_1 \defeq C(d)D^{-d} \quad\text{and}\quad \quad C_2\defeq C(d)D^{-2d+1}\,,
\end{equation}
such that the following holds.
\begin{enumerate}
    \item \label{i:one-to-one} $F$ is one-to-one on $B(0, C_1r)$.
    \item \label{i:inclusion} $F(B(0, C_1r)) \supset B(F(0), C_2r^2)$.
    \item \label{i:det-lb-ub} $\frac{1}{2}r \leq \abs{\det DF(x)} \leq \frac{3}{2}r$ on $B(0, C_1r)$. 
\end{enumerate}
\end{lemma}

\begin{lemma}[Quantitative Implicit Function Theorem]\label{l:quantitative-ift2} 
Let $G:\R^{d_1+d_2}\to \R^{d_2}$ be a $C^2$ function. Suppose that there exist $x_0\in \R^{d_1}, y_0\in \R^{d_2}$ and $r > 0$, $D\geq 1$ such that
\begin{enumerate}
    \item \label{i:Gzero} $G(x_0, y_0)=0$.
    \item \label{i:Drassumption2}
        $ D \geq \norm{G}_{C^2} \quad\text{and}\quad r\leq \abs{\det D_y G(x_0,y_0)}\,. 
        $
\end{enumerate}

Then, there exists constant
\begin{equation}
    C_2\defeq C(d_1+d_2)D^{-2(d_1+d_2)+1}\,,
\end{equation}
such that for some $C^1$ function $H\colon B(x_0, C_2 r^2) \to \R^{d_2}$ and any $x\in B(x_0, C_2r^2)$, 
\begin{enumerate}
    \item \label{i:GHzero} $G(x, H(x))=0$.
    \item \label{i:detGH-lb} $\abs{\det D_y G(x, H(x))} \geq \frac{1}{2}r$.
\end{enumerate}
Here, $C(d_1+d_2)$ is a constant depending on the dimension $d_1+d_2$.
    
\end{lemma}

\begin{proof}[Proof of Lemma~\ref{l:z-smallset}]
Define a map $G\colon \T ^{(2), 2}\times \T^4 \to \T^{(2), 2}$ such that $G(z, \zeta) = \Phi_4^{(2)}(z, \zeta)-\Phi_4^{(2)}(w, \xi)$.
Then clearly, $G(w, \xi)=0$ holds. We choose $D=CA^7, r=s$ so that~\eqref{e:dzeta}--\eqref{e:dxdzeta}, and the assumption (\ref{i:wxi-lb}) imply the condition (\ref{i:Drassumption2}) of Lemma~\ref{l:quantitative-ift2} at $(x_0, y_0)=(w, \xi)$. Then Lemma~\ref{l:quantitative-ift2} implies that there exists a ball $B_1\defeq B(w, CA^{-105}s^2)$ and a $C^1$-function $H\colon B_1 \to \T^4$ such that for any $z \in B_1$,
\begin{equation}\label{e:det-lb}
    \Phi_4^{(2)}(z,H(z))=\Phi_4^{(2)}(w, \xi) \overset{(\ref{i:wxi-fp})}{=} w \quad\text{and}\quad \abs{\det D_\zeta G(z, H(z))}\geq \frac{1}{2}s\,.
\end{equation}

Now, we fix $z\in B_1$ and define a map $F= \Phi_4^{(2)}(z, \cdot+H(z))\colon \T^4 \to \T^4$. We choose $D=CA^7$ and $r=\frac{1}{2}s$ so that~\eqref{e:dzeta}, \eqref{e:dzetadzeta}, and~\eqref{e:det-lb} imply~\eqref{e:Drassumption}, and hence by Lemma~\ref{l:quantitative-ift}, there exists some constants $C, C'$ independent of $z, w, \xi$,  and $A$, such that
\begin{enumerate}
    \item $F$ is one-to-one on $B_2\defeq B(0, CA^{-28}s)$.
    \item $F(B_2) \supset B_3\defeq B(F(0), C'A^{-49}s^2)=B(w, C'A^{-49}s^2)$.
    \item $\abs{\det DF(x)} \leq \frac{3}{4}s$ on $B_2$.
\end{enumerate}
Then,
\begin{align}
    P_A^{(2), 2}(z,E) &= \int_{\T^4} \one_E(F(\zeta)) d\zeta \geq \int_{B_2}\one_E(F(\zeta))d\zeta = \int_{F(B_2)}\one_E(y) \abs{\det D_y F^{-1}(y)} dy\\
    &\geq \int_{B_3} \one_E(y) \abs{\det D_\zeta F(F^{-1}(y))}^{-1} dy \\
    &\geq \frac{4}{3}s^{-1}\Leb|_{B_3}(E)\,.
\end{align}
Decreasing $C$ or $C'$ if necessary, we complete the proof.
\end{proof}

We complement the quantitative inverse function theorem in Section 8 of~\cite{Christ85} 
by explicitly expressing the constants $C_1$ and $C_2$ in terms of $D$. 
For the reader's convenience, we reproduce the proof of the lemma below.

\begin{proof}[Proof of Lemma~\ref{l:quantitative-ift}]
    Throughout the proof, $C(d)$ represents a generic constant whose value may vary from line to line but depend only on the dimension $d$.

    We start by proving the item (\ref{i:det-lb-ub}).
    By the assumption~\eqref{e:Drassumption}, we have
    \begin{equation}
        \norm{D_x \det DF(x)}_\infty \leq C(d) \norm{D_x^2 F}_\infty \norm{D_x F}_\infty^{d-1} \leq C(d)D^d\,.   
    \end{equation}
    Combining this with the property of $r$ in~\eqref{e:Drassumption} and the definition of $C_1$ in~\eqref{e:Cdefs} yields the item (\ref{i:det-lb-ub}). 
    Next we prove item (\ref{i:one-to-one}). We first note that for any matrix $M\in \R^{d\times d}$ and vector $w\in \R^d$, 
    \begin{equation}
        \abs{Mw} \geq (d-1)^\frac{d-1}{4}\frac{\abs{\det M}}{\norm{M}_F^{d-1}}\abs{w}\,,
    \end{equation}
    where $\norm{\cdot}_F$ is the Frobenius norm (See for instance~\cite{YishengDunhe97}).
    Then the lower bound of $\abs{\det DF(x)}$ in the item (\ref{i:det-lb-ub}) implies that for any $x\in B(0, C_1r)$ and a vector $v\in \R^d$,
    \begin{equation}\label{e:dfv-lb}
        \abs{DF(x) v} \geq (d-1)^\frac{d-1}{4}\frac{\abs{\det DF(x)}}{\norm{DF(x)}_F^{d-1}}\abs{v} \geq C_4 r\abs{v}\,,
    \end{equation}
    where
    \begin{equation}\label{e:C4def}
        C_4\defeq C(d)D^{-(d-1)}\,.    
    \end{equation}
    We let $x_1, x_2 \in B(0, C_1r)$ and define $z(t)\defeq F(x_1+t(x_2-x_1)), t\in [0,1]$. Then $z'(t)=DF(x_1+t(x_2-x_1))(x_2-x_1)$ and 
    \begin{align}
        \ip{z'(t) , z'(0)} &\geq \abs{z'(0)}^2 - D\abs{x_2-x_1}^2\abs{z'(0)}\\
        &\overset{\eqref{e:dfv-lb}}{\geq} (C_4r-D\abs{x_2-x_1})\abs{x_2-x_1}\abs{z'(0)}\\
        &\geq (C_4-2DC_1)r\abs{x_2-x_1}\abs{z'(0)} > 0\,,
    \end{align}
    provided $C_1 \leq \frac{1}{2}C_4D^{-1}$. This implies
    \begin{equation}
        \ip{z(1)-z(0), z'(0)}  > 0\,,
    \end{equation}
    and hence $z(1)\neq z(0)$, which completes the proof for item (\ref{i:one-to-one}).
    Now it remains to prove item (\ref{i:inclusion}). We suppose without loss of generality that $F(0)=0$. We fix $v\in S^{d-1}$ and define the vector field $X(x)$ on $B(0, C_1r)$ such that $DF(x)X(x)=v$. Then by the property~\eqref{e:dfv-lb}, 
    \begin{equation}\label{e:Xbound}
        \abs{X(x)}\leq C_4^{-1}r^{-1}\,.   
    \end{equation}
    We consider the ODE
    \begin{equation}
        z'(t) = X(z(t))\,, \quad z(t)\in B(0, C_1 r)\,, \quad z(0)=0\,.
    \end{equation}
    The boound~\eqref{e:Xbound} on $X$ implies that $z(t)$ is well-defined on time $0\leq t\leq C_1C_4r^2$, and also the definition of $X$ implies that
    \begin{equation}
        \frac{d}{dt}F(z(t)) = DF(z(t))X(z(t))=v\,.
    \end{equation}
    Thus,
    \begin{equation}
        F(z(t))=vt\,, \quad t\in [0, C_1C_4r^2]\,.
    \end{equation}
    This implies $C_1C_4r^2 v \in F(B(0, C_1r))$ and hence, using the definitions~\eqref{e:Cdefs} and \eqref{e:C4def} of $C_1$ and $C_4$, respectively, we conclude the proof for the item (\ref{i:inclusion}).
\end{proof}

By applying the quantitative inverse function theorem, we also reproduce the proof of the implicit function theorem, carefully tracing the dependence of the constants in the classical proof.

\begin{proof}[Proof of Lemma~\ref{l:quantitative-ift2}]
    We define a function 
    \begin{equation}
        \Tilde G (x,y)\colon \R^{d_1+d_2} \to \R^{d_1+d_2}, (x,y)\mapsto (x, G(x,y))\,.   
    \end{equation}
    Then by assumption (\ref{i:Drassumption2}), $\norm{\tilde G}_{C^2} \leq \max(1, \norm{G}_{C^2}) \leq D$ and $\abs{\det D\tilde G (x_0, y_0)} = \abs{\det D_y G (x_0, y_0)}\geq r > 0$. This satisfies the assumption for Lemma~\ref{l:quantitative-ift} so that there exist constants $C_1(D)=C(d_1+d_2)D^{-(d_1+d_2)}$ and $C_2(D)=C(d_1+d_2)D^{-2(d_1+d_2)+1}$ such that
    $\tilde G$ is a $C^1$ diffeomorphism between $V$ and $W$, where 
    \begin{align}
    W&\defeq B(\tilde G(x_0, y_0), C_2(D)r^2)\overset{(\ref{i:Gzero})}{=}B((x_0, 0), C_2(D)r^2)\,,\\
    V&\defeq\tilde G^{-1}(W)\cap B((x_0, y_0), C_1(D)r)\,.
    \end{align}
    Moreover, item (\ref{i:det-lb-ub}) implies 
    \begin{equation}\label{e:detDGtilde-lb}
        \abs{\det D\tilde G (x,y)} = \abs{\det D_y G(x,y)} \geq \frac{1}{2}r\,, \quad\forall (x,y) \in V\,.
    \end{equation}
    We define $U\defeq \set{x\in \R^{d_1}\st (x,0)\in W}$ and $H\colon U \to \R^{d_2}, x \mapsto \pi_{d_2}(\tilde G^{-1} (x,0))$, where $\pi_{d_2}:\R^{d_1+d_2} \to \R^{d_2}, (z_1, z_2, \ldots, z_{d_1+d_2})\mapsto (z_{d_1+1}, \ldots, z_{d_1+d_2})$ is the projection onto the last $d_2$ coordinates. Then $U$ is an open neighborhood of $x_0$ in $\R^{d_1}$ and in particular, $B(x_0, C_2(D)r^2) \subset U$. Also, $\tilde G^{-1}$ is $C^1$ on $W$ so that $H$ is $C^1$, and
    \begin{equation}
        \set{(x, H(x))\st x\in U} = \set{(x,y) \in V\st G(x,y)=0}\,.
    \end{equation}
    Then since $\set{(x,H(x))\st x\in B(x_0, C_2(D)r^2)} \subset \set{(x,H(x))\st x\in U}$, $(x,H(x))\in V$ and $G(x,H(x))=0$. Combining this with~\eqref{e:detDGtilde-lb} concludes the proof for both the item (\ref{i:GHzero}) and (\ref{i:detGH-lb}). 
\end{proof}

\section{Proof of Lemma~\ref{l:Lyapunov-drift} and~\ref{l:Harris-conditions}}\label{s:verify-Harris}

In this section, we prove the Harris conditions, namely Lemma~\ref{l:Lyapunov-drift} and Lemma~\ref{l:Harris-conditions}. 
In Section~\ref{s:off-diagonal-small}, we first establish in Lemma~\ref{l:final-small-set} 
that the off-diagonal set is a small set, with explicit constants that depend on the amplitude~\(A\). 
Then, in Section~\ref{s:proof-of-harris-lemma}, we prove Lemma~\ref{l:Lyapunov-drift} 
by adapting results from~\cite{CoopermanIyerSon25}, 
and deduce Lemma~\ref{l:Harris-conditions} 
by combining this with the off-diagonal small set property established in Lemma~\ref{l:final-small-set}.

\subsection{Small Set Property of Off-Diagonal Sets}\label{s:off-diagonal-small}

The main goal of this subsection is to establish the small set property of the off-diagonal set, 
as stated in Lemma~\ref{l:final-small-set}. 

Before presenting the lemmas and proofs, we outline the main idea behind Lemma~\ref{l:final-small-set} 
and explain the role of the intermediate lemmas. To prove the small set property, 
we first show that any sufficiently separated two-point configuration $z = (x, y) \in \T^{2,(2)}$ 
can reach the special set $R$, defined in~\eqref{e:Rdef}, in $N_1(z)$ steps (Lemma~\ref{l:smallset-to-R1}). Using a uniform property of $R$ 
(Lemma~\ref{l:R1-uniform-small}), we then show that any point $z' = (x', y')$ near $R$ 
can reach a fixed small ball $S = B(z_*, r_1)$ in $N_2(x')$ steps (Lemma~\ref{l:R_1-small}). 
Since $S$ admits a minorizing measure supported on a concentric ball $S' = B(z_*, r_2)$, 
there is a positive probability for the chain starting in $S$ to return to $S$. 

Combining these steps, for any configuration in the sublevel set of the Lyapunov function $V$, 
there exists a trajectory that enters and remains in $S$ with positive probability until another trajectory 
enters $S$. After at most $\max_{z, x'} (N_1(z) + N_2(x'))$ steps, we can thus ensure 
that some trajectories starting from the sublevel set are contained in the common ball $S$ 
with positive probability. Applying the small set property of $S$ then 
completes the proof of Lemma~\ref{l:final-small-set}.

We begin by showing that any sufficiently separated two-point configuration \( z \) can reach \( R \) in finite time with positive probability. To state this precisely, let \( r_1 \) be as in Lemma~\ref{l:R1-uniform-small} and define the neighborhood \( R_1 \) of \( R \) by
\begin{equation}\label{e:R1-def}
    R_1 \defeq \bigcup_{z \in R} B_\infty(z, r_1)\,.    
\end{equation}
We also define the near-diagonal set in $\T^{2,(2)}$ as
\begin{equation}\label{e:on-diagonal-def}
    \Delta(s) \defeq \set{(x, y) \in \T^2 \times \T^2 \st 0 < d_{\T^2}(x, y) < s}\,.
\end{equation}

\begin{lemma}\label{l:smallset-to-R1}
    Let $s_*$ be an $A$-independent constant, $z\in \Delta(A^{-2}s_*)^c$, and $r_1$ be as in Lemma~\ref{l:R1-uniform-small}. Then there exists $N_1(z)\in \N$ and $\xi^{N_1}(z)\in \T^{2N_1}$ such that
    \begin{equation}\label{e:zeta-to-smallset}
        \Phi_{2N_1}^{(2)}(z, \xi^{N_1}) \in R\,.
    \end{equation}
    Moreover, for sufficiently large $A$, 
    \begin{gather}
        \label{e:time-to-smallset-ub}
        \sup_{z\in \Delta(A^{-2}s_*)^c} N_1(z) \leq \frac{6\pi}{A\sin(A^{-2}s_*)}\,,\\
        \label{e:to-smallset-lb}
        \inf_{z\in \Delta(A^{-2}s_*)^c} P_A^{(2), N_1}(z, B_\infty(\Phi_{2N_1}^{(2)}(z, \xi^{N_1}), r_1)) \geq p_A\,,
    \end{gather}
    where $p_A =  A^{-CA^2}$ for some $A$-independent constnat $C>0$.
\end{lemma}
\begin{proof}
    Lemma 5.6 and Lemma 5.7 in~\cite{BlumenthalCotiZelatiEA23} implies the existence of $N_1(z)$ and $\xi^{N_1}(z)$ such that~\eqref{e:zeta-to-smallset} holds and also the bound~\eqref{e:time-to-smallset-ub} for $N_1(z)$.
    To prove~\eqref{e:to-smallset-lb}, we use~\eqref{e:dzeta} to see that if $\underline{\zeta}^{N_1} \in B_{\infty}(\xi^{N_1}, A^{-2N_1}r_1)$, then
    \begin{equation}
        \abs[\Big]{\Phi_{2N_1}^{(2)}(z,\xi^{N_1}) - \Phi_{2N_1}^{(2)}(z,\underline{\zeta}^{N_1})}_\infty \leq A^{2N_1} \abs{\xi^{N_1} - \underline{\zeta}^{N_1}}_\infty \leq r_1\,.
    \end{equation}
    This implies that
    \begin{align}
        \P\brak[\Big]{\Phi_{2N_1}^{(2)}(z,\underline{\zeta}^{N_1})\in B_\infty(\Phi_{2N_1}^{(2)}(z, \xi^{N_1}), r_1)} &\geq \P\brak[\Big]{\underline{\zeta}^{N_1} \in B_{\infty}(\xi^{N_1}, A^{-2N_1}r_1)}\\
        &=(A^{-2N_1}r_1)^{2N_1} \overset{\eqref{e:time-to-smallset-ub}}{\geq} p_A\,.
    \end{align}
\end{proof}
After reaching the neighborhood $R_1$ of $R$, we next show that the two-point configuration can reach a fixed small set with a probability that can be explicitly traced in terms of $A$.

To define this fixed small set precisely, we denote the special point
\begin{equation}\label{e:z_*-def}
    z_* \defeq ((0, 0), (\pi, \pi)) \in R\,.
\end{equation}

\begin{lemma} \label{l:R_1-small}
Let $r_1$ be as in Lemma~\ref{l:R1-uniform-small}. Then, for any $z'\in R_1$ and $z=(x,y)\in R$ such that $z' \in B_\infty(z, r_1)$, we have
    \begin{equation}\label{e:R-to-smallset}
        P_A^{(2), 2N_2(x)}(z', B_\infty(z_*, r_1)) \geq (c_2r_1^4)^{N_2(x)}\,,\quad\text{where}\quad N_2(x)=\ceil{\abs{x}/r_1}\,.
    \end{equation}

\end{lemma}
\begin{proof}

Fix $z'=(x',y') \in R_1$ and $z=(x, y)\in R$ such that $z'\in B_\infty(z, r_1)$. Define $l_1$ and $l_2$ as the two line segments joining $x$ and $(z_*)_1 = (0,0)$, $y$ and $(z_*)_2 = (\pi, \pi)$, respectively. 

We find a finite collection of open balls $\set{B_\infty(w_i, r_1)\st w_i\in \T^2}$ that cover $l_1$ such that $w_1=x$, all $w_i$'s are on the segment $l_1$, and $\abs{w_i-w_{i+1}}=r_1$. 

Then since $y-x$ and $(z_*)_2-(z_*)_1$ are parallel by the definition of $R$,  $$\set{T(B_\infty(w_i, r_1))\st w_i \in \T^2}$$ also covers the line segment $l_2$ and all $T(w_i)$'s are on $l_2$. Here, $T:\T^2 \to \T^2, w\mapsto (\pi,\pi)+w$ is the translation map.

By the definition of $w_i$'s, we see that
\begin{align}
    \Leb_{\T^2}(B_\infty(w_i, r_1) \cap B_\infty(w_{i+1}, r_1)) &\geq \Leb_{\T^2}(B(w_i, r_1) \cap B(w_{i+1}, r_1)) \\
    &= \paren[\Big]{\frac{2\pi}{3}-\frac{\sqrt{3}}{2}}r_1^2 \geq r_1^2\,,
\end{align}
and combining this with~\eqref{e:zR-uniform-small} and the fact that
\begin{equation}
 B_\infty(z,r_1)=B_\infty(x,r_1)\times B_\infty(y,r_1)=B_\infty(w_1, r_1)\times T(B_\infty(w_1, r_1))\,,
\end{equation}
and $r_1 \leq r_2$ for $A\geq 1$, we get
\begin{align}\label{e:first-transition}
    &P_A^{(2), 2}((x',y'), B_\infty(w_2, r_1)\times T(B_\infty(w_2, r_1)))\\
    &\geq c_2\Leb_{\T^2}(B_\infty(w_1, r_1)\cap B_\infty(w_2, r_1))\Leb_{\T^2}(T(B_\infty(w_1,r_1))\cap T(B_\infty(w_2, r_1)))\\
    &\geq c_2 r_1^4\,.
\end{align}
Similarly, for each $i\geq 2$, if we let $S_i=B_\infty(w_i, r_1)\times T(B_\infty(w_i, r_1))$, then
\begin{equation}
    \inf_{z''\in S_i} P_A^{(2), 2}(z'', S_{i+1}) \geq c_2r_1^4\,.
\end{equation}

Using the Markov property and the fact that the number of balls required to cover $l_1$ are bounded above by $i\leq \abs{l_1}/r=\abs{x}/r$, we obtain~\eqref{e:R-to-smallset}.
\end{proof}
Combining the preceding results, we are now ready to establish the small set property of the off-diagonal set.
\begin{lemma}\label{l:final-small-set}
    Fix $s_*>0$. Let $r_1, r_2, c_1, c_2$ be as in Lemma~\ref{l:R1-uniform-small} and $p_A$ be as in Lemma~\ref{l:smallset-to-R1}. Define 
    \begin{equation}\label{e:Mdef}
        M=\ceil[\Big]{\frac{6\pi}{A\sin(A^{-2}s_*)}}+2\ceil[\Big]{\frac{2\pi}{r_1}}+2\,.
    \end{equation}
    Then, for sufficiently large $A$, we have
    \begin{equation}\label{e:final-small-set}
        \inf_{z\in \Delta(A^{-2}s_*)^c}P_A^{(2), M}(z, \cdot) \geq p_A (c_2r_1^4)^{\ceil[\big]{\frac{2\pi}{r_1}}}c_1^{\floor{\frac{M}{2}}-1}c_2\Leb(B_\infty(z_*, r_2)\cap \cdot)\,.
    \end{equation}
    Here, $\Delta$ and $z_*$ is defined as in~\eqref{e:on-diagonal-def} and~\eqref{e:z_*-def}, respectively.
\end{lemma}
\begin{proof}
    We split the transitions from $\Delta(A^{-2}s_*)^c$ to the measure $\Leb(B_\infty(z_*, r_2)\cap \cdot)$ into three parts: $\Delta(A^{-2}s_*)^c$ to $R$, $R$ to $B(z_*, r_1)$, and $B(z_*, r_1)$ recurring to itself. If we fix $z=(x,y)\in \Delta(A^{-2}s_*)^c$, define $N_1$ and $N_2$ as in Lemma~\ref{l:smallset-to-R1} and Lemma~\ref{l:R_1-small}, and denote $w=\Phi_{2N_1}^{(2)}(z, \xi^{N_1})=(w_1,w_2)\in R$, then each phase takes $N_1(z), N_2(w_1)$ and $2$ steps. Moreover,
    \begin{align}
        P_A^{(2), N_1(z)+2N_2(w_1)}(z, B_\infty&(z_*, r_1)) \\
        &\geq \int_{B_\infty(w, r_1)} P_A^{(2), 2N_2(w_1)}(z', B_\infty(z_*, r_1)) P_A^{(2), N_1(z)}(z, dz')\\
        &\overset{\eqref{e:to-smallset-lb}, \eqref{e:R-to-smallset}}{\geq} p_A(c_2r_1^4)^{N_2(w_1)}\,.
    \end{align}
    Combining this with the fact that $B_\infty(z_*, r_1)$ is small by Lemma~\ref{l:R1-uniform-small}, we get that for any $n\geq 2$,
    \begin{equation}\label{e:longtime-smallset}
        P_A^{(2), N_1(z)+2N_2(w_1)+n}(z, \cdot) \overset{\eqref{e:zR-uniform-small}}{\geq} p_A(c_2r_1^4)^{N_2(w_1)} c_1^{\floor{\frac{n}{2}}-1} c_2 \Leb|_{B_\infty(z, r_2(A))}(\cdot)\,.
    \end{equation}
    Using the bound~\eqref{e:time-to-smallset-ub} for $N_1(z)$ and the bound
    \begin{equation}\label{e:N2-ub}
        N_2(w_1) \overset{\eqref{e:R-to-smallset}}{\leq} \ceil{2\pi/r_1}\,,
    \end{equation}
    we see that $n=M-N_1(z)-2N_2(w_1) \geq 2$ and using this $n$ in~\eqref{e:longtime-smallset} and the bounds $n\leq M$ and~\eqref{e:N2-ub} again, we deduce~\eqref{e:final-small-set}.
\end{proof}

\subsection{Proof of Lemma~\ref{l:Lyapunov-drift} and Lemma~\ref{l:Harris-conditions}}\label{s:proof-of-harris-lemma}
In this subsection, we prove Lemma~\ref{l:Lyapunov-drift} and Lemma~\ref{l:Harris-conditions}. 
Our construction of the Lyapunov function in Lemma~\ref{l:Lyapunov-drift} is motivated by 
the approach in~\cite[Proposition 9.1]{CoopermanIyerSon25}, but here we explicitly track the 
dependence of the associated constants on the amplitude $A$, completing the proof with the 
identification of an additive constant $K_1$. 
Lemma~\ref{l:Harris-conditions} then follows directly from Lemma~\ref{l:final-small-set}, 
as detailed below.

\begin{proof}[Proof of Lemma~\ref{l:Lyapunov-drift}]
    Proposition 9.1 in~\cite{CoopermanIyerSon25} proves that there exists $\gamma\in (0,1)$, $p\in (0,1)$, and $s_*>0$ such that for all sufficiently large $A>0$, 
    \begin{equation}\label{e:on-diagonal}
        P_A^{(2)}V \leq \gamma V \quad\text{on}\quad \Delta(s_*)\,,
    \end{equation}
    where $\Delta(s_*)$ is defined as in~\eqref{e:on-diagonal-def} and
    $V$ is defined as the continuous extension of
    \begin{equation}\label{e:Vdef}
        \tilde V(x,y) = \abs{x-y}_\infty^{-p} \quad\text{on}\quad \Delta(s_*)\,.
    \end{equation} 
    Then we see that for any $(x, y) \in \Delta(s_*)^c$, 
    \begin{equation}
        \abs{\Phi_2(x)-\Phi_2(y)}_\infty \geq A^{-2}\abs{x-y}_\infty \geq A^{-2}s_* \quad\text{a.s.}\,,
    \end{equation}
    and hence
    \begin{equation}\label{e:off-diagonal}
        P_A^{(2)}V \leq A^{2p}s_*^{-p} \quad\text{on}\quad \Delta(s_*)^c\,,
    \end{equation}
    for sufficiently large $A$.
    Combining~\eqref{e:on-diagonal} and~\eqref{e:off-diagonal} and setting $K_1=s_*^{-p}$ yields~\eqref{e:lyapunov-one-step}.
\end{proof}

\begin{proof}[Proof of Lemma~\ref{l:Harris-conditions}]
For each $n\in \N$, applying $P_A^{(2)}$ for $(n-1)$ times on both sides of the inequality~\eqref{e:lyapunov-one-step} yields that
\begin{equation}
    P_A^{(2), n} V \leq \gamma^n V + K_1A^{2p}\frac{1-\gamma^n}{1-\gamma}\,.
\end{equation}
Setting $n=M$, which we shall choose shortly, immediately implies the first Harris condition~\eqref{e:lyapunov-multi-step}.
To prove the second Harris condition~\eqref{e:sublevel-V-small}, fix a constant $C>1$ and set $R_2 = 2CK_1A^{2p}/(1-\gamma)$. Then the definitions of $V$~\eqref{e:Vdef} and $\Delta$~\eqref{e:on-diagonal-def} imply that there exists an $A$-independent constant $s_*' > 0$ such that 
\begin{equation}\label{e:sub-level-off-diagonal}
    \set{V\leq R_2} \subset \Delta(A^{-2}s_*')^c\,.
\end{equation}
Applying Lemma~\ref{l:final-small-set} to $s_*=s_*'$, setting $M$ as in~\eqref{e:Mdef}, and using the definition of $r_1, r_2, c_1, c_2$, and $p_A$ as in Lemma~\ref{l:R1-uniform-small} and Lemma~\ref{l:smallset-to-R1} imply the bound $M\leq CA^{95}$ and the small set property~\eqref{e:final-small-set}, which yields for all sufficiently large $A>0$,
\begin{equation}\label{e:off-diagonal-small}
    \inf_{z\in \Delta(A^{-2}s_*')^c} P_A^{(2), M}(z, \cdot) \geq A^{-CA^{95}}\nu_A(\cdot)\,, 
\end{equation}
where $\nu_A = (\Leb(B_\infty(z_*, r_2)))^{-1}\Leb|_{B_\infty(z_*, r_2)}$.
Combining~\eqref{e:off-diagonal-small} with~\eqref{e:sub-level-off-diagonal} completes the proof for~\eqref{e:sublevel-V-small}.
\end{proof}

\section{Proof of Theorem~\ref{t:main-theorem}}\label{s:proof-of-main-theorem}

In this section, we complete the proof of Theorem~\ref{t:main-theorem}.  
As is well known from previous works~\cite{BlumenthalCotiZelatiEA23, BedrossianBlumenthalEA21, BedrossianBlumenthalEA22, DolgopyatKaloshinEA04, CoopermanIyerSon25, CoopermanIyerEA25}, the geometric ergodicity of the two-point Markov chain implies almost sure mixing of the one-point chain. To establish this ergodicity and quantify its dependence on the amplitude~\(A\), we apply a quantitative version of the Harris theorem~\cite{HairerMattingly11}. Lemma~\ref{l:Lyapunov-drift} provides a suitable Lyapunov function and tracks the dependence of its constants on~\(A\), while Lemma~\ref{l:Harris-conditions} combines this with the verification of the small set property for an appropriate sublevel set. Together, these results yield the framework required to rigorously control the mixing rate in terms of the flow amplitude.

\begin{proof}[Proof of Theorem~\ref{t:main-theorem}]
    Let $p, \gamma, K_1, V, R_2, M, \nu_A$ be as in Lemma~\ref{l:Harris-conditions}.
    For notational convenience, we define
    \begin{align}
        L_n &\defeq K_1A^{2p}\frac{1-\gamma^n}{1-\gamma}\,,\\
        \alpha &\defeq A^{-CA^{95}}\,,\\
        \gamma_0 &\defeq \frac{3}{4}\,,\\
        \beta &\defeq \frac{\alpha}{2L_M}\,,\\
        \bar \alpha &\defeq \max\set[\Big]{1-\frac{\alpha}{2}, \frac{2+R_2\beta \gamma_0}{2+R_2\beta}}\,,
    \end{align}
    and follow the proof of~\cite[Lemma 3.2, 3.3]{CoopermanIyerSon25}.
    Indeed, $L_M / R_2 = (1-\gamma^M)/4$ so that $\gamma_0=3/4 \in (\frac{1+\gamma^M}{2}, 1) = (\gamma^M + 2\frac{L_M}{R_2}, 1)$ for all sufficiently large $A$. This and Lemma~\ref{l:Harris-conditions} imply that the chain $P_A^{(2), M}$ satisfies the condition for~\cite[Theorem 7.1]{CoopermanIyerSon25} and hence, following the proof of~\cite[Lemma 3.2]{CoopermanIyerSon25}, we obtain that for any $n\in \N$ and mean-zero test function $\varphi$,
    \begin{equation}
        \norm{P_A^{(2), Mn}\varphi}_\beta \leq C\bar \alpha^n \norm{\varphi}_\beta\,,
    \end{equation}
    where the norm $\norm{\cdot}_\beta$ is defined as
    \begin{equation}
        \norm{f}_\beta \defeq \sup_x \frac{\abs{f(x)}}{1+\beta V(x)}\,.
    \end{equation}
    In addition, for all sufficiently large $A>0$,
    \begin{align}
        \abs[\Big]{\frac{P_A^{(2)} \varphi (z)}{1+\beta V(z)}} &\leq \int \abs[\Big]{\frac{\varphi(z')}{1+\beta V(z')}\frac{1+\beta V(z')}{1+\beta V(z)} P_A^{(2)}(z,dz')} \overset{\eqref{e:lyapunov-one-step}}{\leq}  \norm{\varphi}_\beta \frac{1+\beta (\gamma V(z) + L_1)}{1+\beta V(z)} \\
        &= \norm{\varphi}_\beta \paren[\Big]{\gamma + \frac{1-\gamma+\beta L_1}{1+\beta V(z)}} \leq \norm{\varphi}_\beta (1+\beta L_1) \leq \norm{\varphi}_\beta (1+\alpha)\,,
    \end{align}
    which implies
    \begin{equation}
        \norm{P_A^{(2)}\varphi}_\beta \leq (1+\alpha)\norm{\varphi}_\beta\,.
    \end{equation}

    Thus, for all $n\in \N$,
    \begin{equation}    
        \norm{P_A^{(2), n}\varphi}_\beta \leq C (1+\alpha)^M \bar \alpha^{\frac{n}{M}-1} \norm{\varphi}_\beta\,.
    \end{equation}

    $\lim_{A\to \infty}(1+\alpha)^M = 1$ so for all sufficiently large $A$, $(1+\alpha)^M < 2$ and also $\alpha$ is small enough to satisfy
    \begin{equation}
        1- \frac{\alpha}{2}<\frac{2+R_2\beta \gamma_0}{2+R_2\beta} = \frac{3}{4}+\frac{1}{4}\frac{1}{1+(1-\gamma^M)^{-1}\alpha} < 1-\frac{\alpha}{4} \,,
    \end{equation}
    and hence $\frac{1}{2}<\bar{\alpha}=\frac{2+R_2\beta\gamma_0}{2+R_2\beta} < 1-\frac{\alpha}{4}$.

    This implies
    \begin{equation}
        \norm{P_A^{(2), n}\varphi}_\beta \leq C \paren[\bigg]{1-\frac{\alpha}{4}}^\frac{n}{M} \norm{\varphi}_\beta \leq Ce^{-\frac{1}{4} \alpha^2 n}\norm{\varphi}_\beta\,,
    \end{equation}
    for all sufficiently large $A>0$.

   Following the proof of~\cite[Lemma3.3]{CoopermanIyerSon25}, we can choose
   \begin{equation}\label{e:zetadef}
   \zeta \defeq \frac{1}{16}\min\set[\big]{\frac{1}{1+q}, \frac{1}{4}}\alpha^2
   \end{equation}
   for a fixed $q<\infty$ so that $\zeta$ satisfies (8.2)--(8.5), where $d=2, \beta=\frac{1}{4}\alpha^2$. 
   We note that the dependence on $A$ only enters in the last line of the display for the moment bound $\E[\hat D_A^q]$ at the end of the proof, and the inequality immediately preceding (8.1). We track the dependence by examining
   \begin{equation}
      \frac{\zeta q}{\frac{1}{4}\alpha^2 - 2\zeta(1+q)}\,,\quad \norm{e_{m'}^{(2)}}_\beta\,, \quad\text{and}\quad \int (1+\beta V)d\pi^{(2)}\,,
   \end{equation}
   and observe that they satisfy the bound
   \begin{equation}
       \norm{e_{m'}^{(2)}}_\beta\leq 1\,, \quad\text{and}\quad \int (1+\beta V)d\pi^{(2)}\leq 2\,,
   \end{equation}
   for sufficiently large $A>0$, and also the choice of $\zeta$ \eqref{e:zetadef} guarantees
$\frac{\zeta q}{(1/4)\alpha^2 - 2\zeta(1+q)}$ is independent of $A$.
   Then, all the proof will go unchanged for sufficiently large $A>0$ and imply~\eqref{e:mixing} and~\eqref{e:Dmoment-bound}.
\end{proof}

\appendix
\section{Various Estimates for the Flows}

Recall that for any $z\in \T^{2, (2)}$, we denote $z = (x, y)$ for some $x, y \in \T^2$, and for an arbitrary $x \in \T^2$, we write its first and second coordinates as $x_1$ and $x_2$, respectively. For notational brevity, we also suppress the dependence of $\Phi_n^{(2)}$ and $\Phi_n$ on $(\xi_1, \xi_2, \ldots, \xi_n) \in \T^n$ in the notation.

\begin{lemma}
For any $n\geq 2$, there exists some large constant $C(n)>0$, depending on $n$, such that for any $A\geq 1$, $i,i'\in \set{1,\ldots n}$, and $j,j'\in \set{1,2,3,4}$, 
    \begin{align}
        \norm{D_{\xi_i} \Phi_n^{(2)}(z, (\xi_1, \ldots, \xi_n))}_\infty &\leq CA^n\,,\label{e:dzeta}\\
        \norm{D_{z_j} \Phi_n^{(2)}(z, (\xi_1, \ldots, \xi_n))}_\infty &\leq CA^n\,,\label{e:dx}\\
        \norm{D_{\xi_i}D_{\xi_{i'}} \Phi_n^{(2)}(z, (\xi_1, \ldots, \xi_n))}_\infty &\leq CA^{2n-1}\,,\label{e:dzetadzeta}\\
        \norm{D_{z_j}D_{z_{j'}} \Phi_n^{(2)}(z, (\xi_1, \ldots, \xi_n))}_\infty &\leq CA^{2n-1}\,,\label{e:dxdx}\\
        \norm{D_{z_j} D_{\xi_{i}} \Phi_n^{(2)}(z, (\xi_1, \ldots, \xi_n))}_\infty &\leq CA^{2n-1}\,, \label{e:dxdzeta}\\
        %\norm{D_x D_\xi^2 \Phi_n(x)}_\infty &\leq CA^{3(n-1)}\,. \label{e:dxdzetadzeta}
    \end{align}
\end{lemma}

\begin{proof}

We prove the statement by induction. For $n = 1$, the definition of the velocity field $u$ in~\eqref{e:udef} immediately implies
\begin{equation}
    \Phi_1^{(2)}(z) = (\Phi_1(x), \Phi_1(y)) 
    = (x_1 + A \sin(x_2 - \xi_1),\, x_2,\, y_1 + A \sin(y_2 - \xi_1),\, y_2)\,.
\end{equation}
Hence, there exists a constant $C > 0$ such that, for all $A \geq 1$, the estimates~\eqref{e:dzeta}--\eqref{e:dxdzeta} hold with $n = 1$.

Suppose the statement holds with $n=m$ for some $m\in \N$. Without loss of generality, we assume $m$ is even. Then, using the definition of the velocity field $u$ again, we obtain
\begin{equation}\label{e:phi2-recurrence}
    \Phi_{m+1}^{(2)}(z) = \Phi_m^{(2)}(z) + A(\sin(\Phi_m(x)_2 - \xi_{m+1}), 0, \sin(\Phi_m(y)_2 - \xi_{m+1}), 0)\,.
\end{equation}
This implies that for $j, j'\in \set{1,2}$,
\begin{align}
    \label{e:dxj}
    D_{x_j}\Phi_{m+1}^{(2)}(z) &= D_{x_j}\Phi_{m}^{(2)}(z) + A(\cos(\Phi_{m}(x)_2-\xi_{m+1})D_{x_j}\Phi_{m}(x)_2,0,0,0)\,,\\
    \label{e:dxj'dxj}
    D_{x_{j'}}D_{x_j}\Phi_{m+1}^{(2)}(z) &= D_{x_{j'}}D_{x_j}\Phi_{m}^{(2)}(z) + B_1^{j, j'} + B_2^{j, j'}\,,\\
    \label{e:dyj'dxj}
    D_{y_{j'}}D_{x_j} \Phi_{m+1}^{(2)}(z) &= D_{y_{j'}}D_{x_j}\Phi_{m}^{(2)}(z) = \ldots = D_{y_{j'}}D_{x_j}\Phi_1^{(2)}(z) = 0\,,
\end{align}
where
\begin{align}
    \label{e:b1def}
    B_1^{j, j'} &= A(-\sin(\Phi_{m}(x)_2-\xi_{m+1})D_{x_{j'}}\Phi_{m}(x)_2 D_{x_j}\Phi_{m}(x)_2,0,0,0)\,,\\
    \label{e:b2def}
    B_2^{j, j'} &= A(\cos(\Phi_{m}(x)_2-\xi_{m+1})D_{x_{j'}}D_{x_j}\Phi_{m}(x)_2,0,0,0)\,.
\end{align}
We notice that using these computations and the induction hypothesis yields 
\begin{align}
    \norm{D_{x_j}\Phi_{m+1}^{(2)}}_\infty &\leq (1+A)\norm{D_{x_j}\Phi_{m}^{(2)}}_\infty \leq CA^{m+1}\,,\\
    \norm{D_{x_{j'}} D_{x_j}\Phi_{m+1}^{(2)}}_\infty &\leq (1+A)\norm{D_{x_{j'}} D_{x_j}\Phi_{m}^{(2)}}_\infty+A\norm{D_{x_{j'}}\Phi_{m}^{(2)}}_\infty\norm{D_{x_j}\Phi_{m}^{(2)}}_\infty \\
    &\leq CA^{2m} + CA^{2m+1} \leq CA^{2m+1}\,,
\end{align}
for all $A\geq 1$.
$D_{y_j}\Phi_{m+1}^{(2)}$ and $D_{y_{j'}}D_{y_j}\Phi_{m+1}^{(2)}$ cases are similar so combining this with~\eqref{e:dyj'dxj} completes the proof for~\eqref{e:dx} and~\eqref{e:dxdx}.

Similarly, to prove~\eqref{e:dzeta} and~\eqref{e:dzetadzeta}, we differentiate both sides of~\eqref{e:phi2-recurrence} with resepct to $\xi$ and see that for all $i, i'\in \set{1,\ldots, m}$,
\begin{align}
    \label{e:dzetai-m+1}
    D_{\xi_{i}}\Phi_{m+1}^{(2)}(z) &= D_{\xi_{i}}\Phi_{m}^{(2)}(z) + B_3^i\,, \\
    \label{e:dzetam+1-m+1}
     D_{\xi_{m+1}}\Phi_{m+1}^{(2)}(z) &= B_4\,,\\
     D_{\xi_{i'}}D_{\xi_{i}}\Phi_{m+1}^{(2)}(z) &= D_{\xi_{i'}}D_{\xi_{i}}\Phi_{m}^{(2)}(z) + B_5^{i, i'} + B_6^{i, i'}\,,\\
     D_{\xi_i}D_{\xi_{m+1}}\Phi_{m+1}^{(2)}(z) &= B_7^i\,,\\
     D_{\xi_{m+1}}D_{\xi_{m+1}}\Phi_{m+1}^{(2)}(z) &= B_8\,,
\end{align}
where
\begin{align}
    B_3^i &= A(\cos(\Phi_{m}(x)_2-\xi_{m+1})D_{\xi_i}\Phi_{m}(x)_2,0,
               \cos(\Phi_{m}(y)_2-\xi_{m+1})D_{\xi_i}\Phi_{m}(y)_2,0)\,,\\
    B_4 &= -A(\cos(\Phi_{m}(x)_2-\xi_{m+1}),0,
               \cos(\Phi_{m}(y)_2-\xi_{m+1}),0)\,,\\[4pt]
    B_5^{i,i'} &= -A\big(
        \sin(\Phi_{m}(x)_2-\xi_{m+1})
        D_{\xi_{i'}}\Phi_{m}(x)_2D_{\xi_i}\Phi_{m}(x)_2,0, \\
        &\hspace{4em}
        \sin(\Phi_{m}(y)_2-\xi_{m+1})
        D_{\xi_{i'}}\Phi_{m}(y)_2D_{\xi_i}\Phi_{m}(y)_2,0
    \big)\,,\\[2pt]
    B_6^{i,i'} &= A\big(
        \cos(\Phi_{m}(x)_2-\xi_{m+1})
        D_{\xi_{i'}}D_{\xi_i}\Phi_{m}(x)_2,0, \\
        &\hspace{4em}
        \cos(\Phi_{m}(y)_2-\xi_{m+1})
        D_{\xi_{i'}}D_{\xi_i}\Phi_{m}(y)_2,0
    \big)\,,\\[2pt]
    B_7^i &= A(\sin(\Phi_{m}(x)_2-\xi_{m+1})D_{\xi_i}\Phi_{m}(x)_2,0,
               \sin(\Phi_{m}(y)_2-\xi_{m+1})D_{\xi_i}\Phi_{m}(y)_2,0)\,,\\
    B_8 &= -A(\sin(\Phi_{m}(x)_2-\xi_{m+1}),0,
              \sin(\Phi_{m}(y)_2-\xi_{m+1}),0)\,.
\end{align}

Using these computations and the induction hypothesis, we obtain that for all $i,i' \in \set{1, \ldots, m}$ and $A \geq 1$,
\begin{align}
    \norm{D_{\xi_{i}}\Phi_{m+1}^{(2)}}_\infty &\leq (1+A)\norm{D_{\xi_{i}}\Phi_{m}^{(2)}}_\infty \leq CA^{m+1}\,,\\
    \norm{D_{\xi_{m+1}}\Phi_{m+1}^{(2)}}_\infty &\leq A\,,\\
    \norm{D_{\xi_{i'}}D_{\xi_{i}}\Phi_{m+1}^{(2)}}_\infty &\leq (1+A)\norm{D_{\xi_{i'}}D_{\xi_{i}}\Phi_{m}^{(2)}}_\infty + A\norm{D_{\xi_{i'}}\Phi_{m}^{(2)}}_\infty\norm{D_{\xi_{i}}\Phi_{m}^{(2)}}_\infty \\
    &\leq CA^{2m} + CA^{2m+1} \leq CA^{2m+1}\,,\\
    \norm{D_{\xi_{i}}D_{\xi_{m+1}}\Phi_{m+1}^{(2)}}_\infty &\leq A\norm{D_{\xi_{i}}\Phi_{m}^{(2)}}_\infty \leq CA^{m+1}\,,\\
    \norm{D_{\xi_{m+1}}D_{\xi_{m+1}}\Phi_{m+1}^{(2)}}_\infty &\leq A\,,
\end{align}
and hence~\eqref{e:dzeta} and~\eqref{e:dzetadzeta} still hold for $n=m+1$.
Finally, to prove~\eqref{e:dxdzeta}, we differentiate~\eqref{e:dzetai-m+1} and~\eqref{e:dzetam+1-m+1} with respect to $x_j$, $j\in\set{1,2}$. We compute that for all $i\in \set{1,\ldots, m}$ and $j\in \set{1,2}$,
\begin{align}
    D_{x_j}D_{\xi_{i}}\Phi_{m+1}^{(2)}(z) &= D_{x_{j}}D_{\xi_{i}}\Phi_{m}^{(2)}(z) + B_9^{i, j} + B_{10}^{i,j}\,,\\
    D_{x_j}D_{\xi_{m+1}}\Phi_{m+1}^{(2)}(z) &= B_{11}^{j}\,,
\end{align}
where
\begin{align}
    B_9^{i, j} &= -A(\sin(\Phi_{m}(x)_2-\xi_{m+1})D_{x_j}\Phi_{m}(x)_2 D_{\xi_i}\Phi_{m}(x)_2,0,0,0)\,,\\
    B_{10}^{i, j} &= A(\cos(\Phi_{m}(x)_2-\xi_{m+1})D_{x_j}D_{\xi_i}\Phi_{m}(x)_2,0,0,0)\,,\\
    B_{11}^j &= A(\sin(\Phi_{m}(x)_2-\xi_{m+1})D_{x_j}\Phi_{m}(x)_2,0,0,0)\,.
\end{align}
We deduce that for all $A\geq 1$,
\begin{align}
    \norm{D_{x_j}D_{\xi_{i}}\Phi_{m+1}^{(2)}}_\infty &\leq (1+A)\norm{D_{x_j}D_{\xi_{i}}\Phi_{m}^{(2)}}_\infty + A \norm{D_{x_j}\Phi_{m}^{(2)}}_\infty\norm{D_{\xi_{i}}\Phi_{m}^{(2)}}_\infty\\
    &\leq CA^{2m} + CA^{2m+1} \leq CA^{2m+1}\,,\\
    \norm{D_{x_j}D_{\xi_{m+1}}\Phi_{m+1}^{(2)}}_\infty &\leq A\norm{D_{x_j}\Phi_{m}^{(2)}}_\infty \leq CA^{m+1}\,.
\end{align}
$D_{y_j}D_{\xi_i}\Phi_{m+1}^{(2)}$ and $D_{y_j}D_{\xi_{m+1}}\Phi_{m+1}^{(2)}$ cases are similar so this completes the proof for~\eqref{e:dxdzeta}.
\end{proof}

\begin{lemma}\label{l:det-lb}
Let $z_*=((0,0),(\pi, \pi))$ and $\xi_*^2=(0,0,0,0)$. Then,
\begin{equation}\label{e:det-dzeta-z*}
    \abs{\det D_\zeta \Phi_4^{(2)}(z_*, \xi_*^2)} = 4A^6\,.
\end{equation}

\end{lemma}
\begin{proof}
    By using the definition of the Pierrehumbert flow~\eqref{e:udef}, we explicitly compute
    \begin{equation}
        D_\zeta \Phi_4^{(2)}(z_*, \xi_*^2) =
        \begin{pmatrix}
           -A^3-A & -A^2 & -A & 0\\
           -A^4-2A^2 & -A^3-A & -A^2 & -A\\
           A^3+A & -A^2 & A & 0\\
           -A^4-2A^2 & A^3+A & -A^2 & A 
        \end{pmatrix}\,,
    \end{equation}
    and its determinant
    \begin{equation}
        \det D_\zeta \Phi_4^{(2)}(z_*, \xi_*^2) = -4A^6\,.
    \end{equation} 
\end{proof}

\bibliographystyle{halpha-abbrv}
\bibliography{gautam-refs1,gautam-refs2,preprints}

\begin{thebibliography}{ABBea17}
\expandafter\ifx\csname url\endcsname\relax
  \def\url#1{\texttt{#1}}\fi
\expandafter\ifx\csname doi\endcsname\relax
  \def\doi#1{\burlalt{doi:#1}{http://dx.doi.org/#1}}\fi
\expandafter\ifx\csname urlprefix\endcsname\relax\def\urlprefix{URL }\fi
\expandafter\ifx\csname href\endcsname\relax
  \def\href#1#2{#2}\fi
\expandafter\ifx\csname burlalt\endcsname\relax
  \def\burlalt#1#2{\href{#2}{#1}}\fi

\bibitem[ABBea17]{ArefEA17}
H.~Aref, J.~R. Blake, M.~Budi\v{s}i\'{c}, and et~al.
\newblock Frontiers of chaotic advection.
\newblock {\em Rev. Modern Phys.}, 89(2):025007, 66, 2017.
\newblock \doi{10.1103/RevModPhys.89.025007}.

\bibitem[ACM19]{AlbertiCrippaEA19}
G.~Alberti, G.~Crippa, and A.~L. Mazzucato.
\newblock Exponential self-similar mixing by incompressible flows.
\newblock {\em J. Amer. Math. Soc.}, 32(2):445--490, 2019.
\newblock \doi{10.1090/jams/913}.

\bibitem[BBPS21]{BedrossianBlumenthalEA21}
J.~Bedrossian, A.~Blumenthal, and S.~Punshon-Smith.
\newblock Almost-sure enhanced dissipation and uniform-in-diffusivity exponential mixing for advection-diffusion by stochastic {N}avier-{S}tokes.
\newblock {\em Probab. Theory Related Fields}, 179(3-4):777--834, 2021.
\newblock \doi{10.1007/s00440-020-01010-8}.

\bibitem[BBPS22]{BedrossianBlumenthalEA22}
J.~Bedrossian, A.~Blumenthal, and S.~Punshon-Smith.
\newblock Almost-sure exponential mixing of passive scalars by the stochastic {N}avier-{S}tokes equations.
\newblock {\em Ann. Probab.}, 50(1):241--303, 2022.
\newblock \doi{10.1214/21-aop1533}.

\bibitem[BCZG23]{BlumenthalCotiZelatiEA23}
A.~Blumenthal, M.~Coti~Zelati, and R.~S. Gvalani.
\newblock Exponential mixing for random dynamical systems and an example of {P}ierrehumbert.
\newblock {\em Ann. Probab.}, 51(4):1559--1601, 2023.
\newblock \doi{10.1214/23-aop1627}.

\bibitem[BH17]{BedrossianHe17}
J.~Bedrossian and S.~He.
\newblock Suppression of blow-up in {P}atlak-{K}eller-{S}egel via shear flows.
\newblock {\em SIAM J. Math. Anal.}, 49(6):4722--4766, 2017.
\newblock \doi{10.1137/16M1093380}.

\bibitem[BVO09]{Beron2009}
F.~Beron-Vera and M.~Olascoaga.
\newblock An assessment of the importance of chaotic stirring and turbulent mixing on the west florida shelf.
\newblock {\em Journal of physical oceanography}, 39(7):1743--1755, 2009.

\bibitem[CFIN23]{ChristieFengEA23}
A.~Christie, Y.~Feng, G.~Iyer, and A.~Novikov.
\newblock Speeding up {L}angevin dynamics by mixing, 2023, \burlalt{2303.18168}{http://arxiv.org/abs/2303.18168}.

\bibitem[Chr85]{Christ85}
M.~Christ.
\newblock Hilbert transforms along curves: I. nilpotent groups.
\newblock {\em Annals of Mathematics}, 122(3):575--596, 1985.
\newblock \urlprefix\url{http://www.jstor.org/stable/1971330}.

\bibitem[CIRS25]{CoopermanIyerEA25}
W.~Cooperman, G.~Iyer, K.~Rowan, and S.~Son.
\newblock Exponentially mixing flows with slow enhanced dissipation, 2025, \burlalt{2507.21305}{http://arxiv.org/abs/2507.21305}.
\newblock \urlprefix\url{https://arxiv.org/abs/2507.21305}.

\bibitem[CIS25]{CoopermanIyerSon25}
W.~Cooperman, G.~Iyer, and S.~Son.
\newblock A {H}arris theorem for enhanced dissipation, and an example of {P}ierrehumbert.
\newblock {\em Nonlinearity}, 38(4):Paper No. 045027, 32, 2025.
\newblock \doi{10.1088/1361-6544/adc652}.

\bibitem[CKRZ08]{ConstantinKiselevEA08}
P.~Constantin, A.~Kiselev, L.~Ryzhik, and A.~Zlato{\v{s}}.
\newblock Diffusion and mixing in fluid flow.
\newblock {\em Ann. of Math. (2)}, 168(2):643--674, 2008.
\newblock \doi{10.4007/annals.2008.168.643}.

\bibitem[Coo23]{Cooperman23}
W.~Cooperman.
\newblock Exponential mixing by shear flows.
\newblock {\em SIAM Journal on Mathematical Analysis}, 55(6):7513--7528, 2023, \burlalt{https://doi.org/10.1137/22M1513861}{http://arxiv.org/abs/https://doi.org/10.1137/22M1513861}.
\newblock \doi{10.1137/22M1513861}.

\bibitem[CRWZ23]{ChengRajasekaranEA23}
L.-T. Cheng, F.~Rajasekaran, K.~Y.~J. Wong, and A.~Zlato\v{s}.
\newblock Numerical evidence of exponential mixing by alternating shear flows.
\newblock {\em Commun. Math. Sci.}, 21(2):529--541, 2023.
\newblock \doi{10.4310/cms.2023.v21.n2.a10}.

\bibitem[CZDE20]{CotiZelatiDelgadinoEA20}
M.~Coti~Zelati, M.~G. Delgadino, and T.~M. Elgindi.
\newblock On the relation between enhanced dissipation timescales and mixing rates.
\newblock {\em Comm. Pure Appl. Math.}, 73(6):1205--1244, 2020.
\newblock \doi{10.1002/cpa.21831}.

\bibitem[DKK04]{DolgopyatKaloshinEA04}
D.~Dolgopyat, V.~Kaloshin, and L.~Koralov.
\newblock Sample path properties of the stochastic flows.
\newblock {\em Ann. Probab.}, 32(1A):1--27, 2004.
\newblock \doi{10.1214/aop/1078415827}.

\bibitem[ELM25]{ElgindiLissEA25}
T.~M. Elgindi, K.~Liss, and J.~C. Mattingly.
\newblock Optimal enhanced dissipation and mixing for a time-periodic, {L}ipschitz velocity field on {$\Bbb{T}^2$}.
\newblock {\em Duke Math. J.}, 174(7):1209--1260, 2025.
\newblock \doi{10.1215/00127094-2024-0057}.

\bibitem[EZ19]{ElgindiZlatos19}
T.~M. Elgindi and A.~Zlato\v{s}.
\newblock Universal mixers in all dimensions.
\newblock {\em Adv. Math.}, 356:106807, 33, 2019.
\newblock \doi{10.1016/j.aim.2019.106807}.

\bibitem[FI19]{FengIyer19}
Y.~Feng and G.~Iyer.
\newblock Dissipation enhancement by mixing.
\newblock {\em Nonlinearity}, 32(5):1810--1851, 2019.
\newblock \doi{10.1088/1361-6544/ab0e56}.

\bibitem[He23]{He23}
S.~He.
\newblock Enhanced dissipation and blow-up suppression in a chemotaxis-fluid system.
\newblock {\em SIAM J. Math. Anal.}, 55(4):2615--2643, 2023.
\newblock \doi{10.1137/22M1517159}.

\bibitem[HM11]{HairerMattingly11}
M.~Hairer and J.~C. Mattingly.
\newblock Yet another look at {H}arris' ergodic theorem for {M}arkov chains.
\newblock In {\em Seminar on {S}tochastic {A}nalysis, {R}andom {F}ields and {A}pplications {VI}}, volume~63 of {\em Progr. Probab.}, pages 109--117. Birkh\"{a}user/Springer Basel AG, Basel, 2011.
\newblock \doi{10.1007/978-3-0348-0021-1\_7}.

\bibitem[IKX14]{IyerKiselevEA14}
G.~Iyer, A.~Kiselev, and X.~Xu.
\newblock Lower bounds on the mix norm of passive scalars advected by incompressible enstrophy-constrained flows.
\newblock {\em Nonlinearity}, 27(5):973--985, 2014.
\newblock \doi{10.1088/0951-7715/27/5/973}.

\bibitem[IXZ21]{IyerXuEA21}
G.~Iyer, X.~Xu, and A.~Zlato\v{s}.
\newblock Convection-induced singularity suppression in the {K}eller-{S}egel and other non-linear {PDE}s.
\newblock {\em Trans. Amer. Math. Soc.}, 374(9):6039--6058, 2021.
\newblock \doi{10.1090/tran/8195}.

\bibitem[KX16]{KiselevXu16}
A.~Kiselev and X.~Xu.
\newblock Suppression of chemotactic explosion by mixing.
\newblock {\em Arch. Ration. Mech. Anal.}, 222(2):1077--1112, 2016.
\newblock \doi{10.1007/s00205-016-1017-8}.

\bibitem[LTD11]{LinThiffeaultEA11}
Z.~Lin, J.-L. Thiffeault, and C.~R. Doering.
\newblock Optimal stirring strategies for passive scalar mixing.
\newblock {\em J. Fluid Mech.}, 675:465--476, 2011.
\newblock \doi{10.1017/S0022112011000292}.

\bibitem[MD18]{MilesDoering18}
C.~J. Miles and C.~R. Doering.
\newblock Diffusion-limited mixing by incompressible flows.
\newblock {\em Nonlinearity}, 31(5):2346, 2018.
\newblock \doi{10.1088/1361-6544/aab1c8}.

\bibitem[Pie94]{Pierrehumbert94}
R.~T. Pierrehumbert.
\newblock Tracer microstructure in the large-eddy dominated regime.
\newblock {\em Chaos, Solitons \& Fractals}, 4(6):1091--1110, 1994.
\newblock \doi{10.1016/0960-0779(94)90139-2}.

\bibitem[Sei13]{Seis13}
C.~Seis.
\newblock Maximal mixing by incompressible fluid flows.
\newblock {\em Nonlinearity}, 26(12):3279--3289, 2013.
\newblock \doi{10.1088/0951-7715/26/12/3279}.

\bibitem[TZ25]{TaoZworski25}
Z.~Tao and M.~Zworski.
\newblock Optimal enhanced dissipation for contact {A}nosov flows.
\newblock {\em Nonlinearity}, 38(4):Paper No. 045001, 9, 2025.
\newblock \doi{10.1088/1361-6544/adbb4b}.

\bibitem[YSDh97]{YishengDunhe97}
Y.~Yi-Sheng and G.~Dun-he.
\newblock A note on a lower bound for the smallest singular value.
\newblock {\em Linear Algebra and its Applications}, 253(1):25--38, 1997.
\newblock \doi{https://doi.org/10.1016/0024-3795(95)00784-9}.

\bibitem[YZ17]{YaoZlatos17}
Y.~Yao and A.~Zlato\v{s}.
\newblock Mixing and un-mixing by incompressible flows.
\newblock {\em J. Eur. Math. Soc. (JEMS)}, 19(7):1911--1948, 2017.
\newblock \doi{10.4171/JEMS/709}.

\end{thebibliography}
\end{document}